\documentclass[11pt]{amsart}
\usepackage{graphicx}
\usepackage{comment}
\usepackage{color}
\newtheorem{thm}{Theorem}[section]
\newtheorem{cor}[thm]{Corollary}
\newtheorem{lem}[thm]{Lemma}

\theoremstyle{definition}
\newtheorem{defn}[thm]{Definition}
\theoremstyle{remark}
\newtheorem{rem}[thm]{Remark}
\numberwithin{equation}{section}
\newcommand{\T}{T}

\newcommand{\Tr}{Tr}

\begin{document}

\title[]{The Non-Euclidean Euclidean Algorithm}

\author{Jane Gilman}
\address{Mathematics Department, Rutgers University, Newark, NJ 07102}
\email{gilman@rutgers.edu}%

%\thanks{ }%
%\subjclass{}%

\dedicatory{Dedicated to the memory of F.W. Gehring}%

% ----------------------------------------------------------------
%\maketitle

\begin{abstract}
In this paper we demonstrate how the geometrically motivated  algorithm to determine whether a two generator real M{\"o}bius group acting on the Poincar{\'e} plane is or is not discrete can be interpreted as a {\sl non-Euclidean Euclidean algorithm}. That is, the algorithm can be viewed as an application of the Euclidean division algorithm to real numbers that represent hyperbolic distances. In the case that the group is discrete and free, the algorithmic procedure also gives a non-Euclidean Euclidean algorithm to find the three shortest curves on the corresponding quotient surface.
\end{abstract}

\maketitle

\section{Introduction}

The problem of determining whether a two generator real M{\"o}bius group acting on the Poincar{\'e} plane (hyperbolic two-space)  is or is  not discrete  is an old one. If such a group is non-elementary and discrete, it is, of course, a Fuchsian group. There are many approaches to this problem, some of them incomplete. One of the most complete answers is given by the use of an algorithm. The algorithm can be given in a number of different forms \cite{G1}, including a geometric form and an algebraic form.  Revisiting the algorithm has been productive, as the algorithm has been shown to have  a number of  useful implications, including results about primitives and palindromes in rank two free groups,  discreteness criteria for complex M\"obius groups acting on hyperbolic three space, and the computational complexity of the discreteness problem \cite{G, G1, GKwords,  GKgeom,  GKEnum, Gx, YCJ}.
In this paper we revisit the  Gilman-Maskit  algorithm
     \cite{GM} in the case of a pair of hyperbolic generators $A$ and $B$ with disjoint axes  and illustrate that it is a type of {\sl non-Euclidean Euclidean algorithm}.
We refer
to the Gilman-Maskit geometric form of the algorithm as the GM algorithm or as the geometric algorithm or simply as the algorithm.

             An algorithm is termed a non-Euclidean Euclidean algorithm if it involves performing Euclidean algorithm type calculations to quantities that are non-Euclidean lengths, in particular to the
translation lengths,  $T_A$ and $T_B$,  of the isometries when they (or a conjugate pair)  act as hyperbolic isometries on the Poincar{\'e} plane.

The algorithms in \cite{GM,G2}  can also be viewed as algorithms to find the three shortest geodesic  when the group is discrete. The paper \cite{G2} addresses the case of intersecting hyperbolic axes.
We point out how the concept of the non-Euclidean Euclidean algorithm applies to all types of pairs of isometries acting on hyperbolic two-space as such a shortest length algorithm  even where no algorithm is needed to determine discreteness (e.g. as in the case of two hyperbolic generators with intersecting axes and a hyperbolic commutator). This is done in section \ref{section:all}.

A quick heuristic way to  understand the term {\sl Non-Euclidean Euclidean algorithm} (an NEE algorithm) is by referring to
the description\footnote{After this interpretation of the algorithm as a NEE algorithm was in place, but before the manuscript was written, the author saw this description in a lecture by Hoban.}
 given by Ryan Hoban \cite{Ryan} (see Figure~\ref{fig:ManShadow}).
\begin{figure}[t] \begin{center}
\includegraphics[height=5cm]{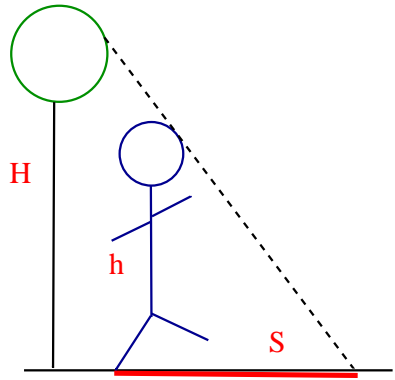}
\caption{In calculus one often asks students to analyze the picture of a person whose shadow is being cast  by
a lamppost. The analysis can be requested under varying rates (e.g. the  rate that the person walks away from the lamppost or the rate that shadow grows) with different quantities shown in the figure known and unknown. Now analyze this situation  when the distances are all hyperbolic distances and the rates are also rates of change of the hyperbolic quantities.}
\label{fig:ManShadow} \end{center}
 \end{figure}

\subsection{The organization of the paper} The main theorems, Theorem  \ref{theorem:new} and its companion Theorem \ref{theorem:pars},  are stated in section \ref{section:main} using a minimal amount of notation and background material. In fact although some notation and terminology is given in section \ref{section:NOTe},  most of section \ref{section:main} can be read without this. In section \ref{section:NOT} further notation is set and terminology reviewed. This section includes subsections on  factoring and the algorithmic paths. In section  \ref{section:barebones}
we give a barebones description of the geometric algorithm
together with figures illustrating the geometry of some of the cases we
need to consider.  The proofs of the main theorems along with the relevant lemmas and their proofs are in section \ref{section:mainlemma}. At the end of this section Theorem  \ref{theorem:new} and its companion Theorem \ref{theorem:pars} are combined into a single theorem,   Theorem \ref{theorem:HP}.
 In section \ref{section:analogy} the analogy is explained and in section \ref{section:all} elliptic elements are addressed and the result on shortest curves, Corollary \ref{cor:shortest}, is stated and proved.

\section{Notation and Terminology} \label{section:NOTe}

 We begin by summarizing standard results about M\"obius transformations acting on the Poincar{\'e} plane (see \cite{Beard}).

Let $X$ be a M\"obius transformation in $PSL(2,\mathbb{R})/\pm id$ and $\tilde{X}$ one of the two matrices in $SL(2,\mathbb{R})$ that projects to $X$. $\tilde{X}$ is termed a pull back of $X$ or equivalently a  lift of $X$.

Now $\tilde{X}$ acts as a fractional linear transformation on the extended complex plane $\hat{\mathbb{C}}$. Since $\tilde{X}$ and its negative induce the same action as fractional linear transformations, $X$ also acts as a fractional linear transformation.

Elements of $SL(2,\mathbb{R})$ fix the upper-half plane, ${{\mathcal{U}}}= \{ z = x+ iy \; | \; y > 0 \}$  and act as isometries when ${\mathcal{U}}$ is endowed with the with the Poincar{\'e} metric. Similarly complex fractional linear transformations that fix the unit disc, $\mathbb{D}$,   act as isometries on the unit disc model for hyperbolic space. Now   ${\mathcal{U}}$ is mapped onto $\mathbb{D}$ by the transformation
$z \mapsto {\frac{z-i}{z + i}}$ and this transformation induces a map from the Poincar{\'e} metric on the upper-half-plane to the Poincar{\'e} metric on the unit disc. Conjugation by the matrix
${\frac{1}{\sqrt{2i}}}
\left(                                                                                       \begin{array}{cc}
                                                                                         1 & -i \\
                                                                                         1 & i \\
                                                                                       \end{array}
                                                                                     \right)
$ sends the matrices in $SL(2,\mathbb{R})$ to matrices that give isometries in the unit disc model and preserves traces.

           Since $X$ and ${\tilde{X}}$ induce the same isometry, we use them interchangeably.

           As  conjugates of $X$ and ${\tilde{X}}$  act as isometries in the unit disc model, we move from the upper-half-plane model to the unit disc model according to convenience using $X$  for the element acting in either model and $\rho$ for the metric in either model.

 We remind the reader of the classification of isometries acting on ${\mathcal{U}}$ {\sl algebraically} by the  absolute value of the  trace of the pull back.
  The matrix ${\tilde{X}}$ and the isometry induced by $X$ or $\tilde{X}$ is hyperbolic, elliptic or parabolic according to whether $|\Tr(\tilde{X})|$, the absolute value of the  trace of the matrix,  is $  > 2, \;\;  < 2 \;\; \mbox{or} \;\;  = 2$.

  If we begin with $A$ and $B$ in $PSL(2,\mathbb{R})$ and choose $\tilde{A}$ and $\tilde{B}$ to have positive trace, then the trace of every element $\tilde{X}$ in the group they generate is determined and we set $\Tr(X) = \Tr ({\tilde{X}})$. A transformation fixing $\mathbb{D}$ is hyperbolic, parabolic or elliptic according to whether its conjugate acting on ${\mathcal{U}}$ is.

If $X$ is hyperbolic, it fixes two points  on the boundary of the ${\mathcal{U}}$  and the hyperbolic geodesic connecting the two fixed points known as the axis of $X$ and denoted $Ax_X$. One of the fixed points is attracting and one is repelling. This means  that $X$ moves points along the axis of $X$ toward the attracting fixed point and it moves all points on $Ax_X$ a fixed distance in the non-Euclidean metric, its translation length, $\T_X$. The same statements apply to the conjugate of $X$ acting on $\mathbb{D}$.

Any two geodesics have a unique common perpendicular. If a geodesic is oriented, then with respect to the positive orientation, there is a left and a right of the geodesic.
\begin{defn}
  Conjugate so that $A$ and $B$ fix $\mathbb{D}$. We can orient $L$,  the common perpendicular to the axes of $A$ and $B$,  so that the positive direction is from the axis of $A$ towards the axis of $B$ (i.e.  moving in the positive direction along $L \cap Ax_A$ is encountered before $L \cap Ax_B$). If $\Tr(A) \ge \Tr(B) > 0$ and if the attracting fixed points of $A$ and $B$ lie to the left of $L$, then we say that the ordered pair $(A,B)$ is  {\sl coherently oriented}. After interchanging $A$ and $B$ and replacing $A$ and/or $B$ by their inverses, we may always assume that we have a coherently oriented pair. Given $A$ and $B$ exactly one of $(A^{\pm 1}, B^{\pm 1})$ and $(B^{\pm 1},A^{\pm 1})$ is a coherently oriented pair. We term this  the {\sl related coherently oriented pair}.
\end{defn}
\section{The main result}\label{section:main}

We let $G$ be the group generated by $A$ and $B$ both elements of $PSL(2,\mathbb{R})$ and assume throughout this paper that $G$ is non-elementary.

The results of \cite{GM, G2, G1, GKwords} can be summarized as
\begin{thm} \label{theorem:GMalg} {\rm \cite{GM, G2} {\bf (Geometric Algorithm)}}
 Let $A$ and $B$ be hyperbolic elements of $PSL(2,\mathbb{R})$. Interchange  $A$ and $B$ so that $\Tr(A) \ge \Tr(B) \ge 2$ and replacing $A$ and/or $B$ its inverse if necessary to assume that the pair $(A,B)$ is coherently oriented. There is a an integer $t$ and a set of positive integers
$[n_1,...,n_t]$ such that replacing the ordered pair $(A,B)$ by the sequence of ordered pairs of generators:
$$(A,B) \rightarrow  (B^{-1}, A^{-1}B^{n_1}) \rightarrow (B^{-n_1}A,
B(A^{-1}B^{n_1})^{n_2}) \rightarrow \ldots (C,D)$$
 after $t$ steps gives an ordered  pair of {\sl stopping generators}  $(C,D)$ and
 outputs $G$ is \\
(i)   {\sl discrete} \\
(ii)  {\sl not discrete}, or  \\
(iii) {\sl not free}.
\end{thm}
We note that in the case that $G$ is not free, the GM algorithm determines discreteness or non-discreteness and finds stopping generators, but the sequence of integers $[n_1,...,n_t]$ must be modified (see section \ref{section:all}).
The sequence $[n_1,...,n_t]$ is used in the calculation of the computational complexity  \cite{G1,YCJ}.
\begin{defn} [\cite{G1,G, GKwords, YCJ}]
The sequence $[n_1,...,n_t]$ is termed the {F-sequence} or the
{\sl Fibonacci sequence} of the algorithm.
\end{defn}
\begin{rem}
Note that in the proof that there is a geometric  algorithm to determine discreteness or non-discreteness (\cite{GM}), the existence of such integers $n_i$ is demonstrated, but the integers themselves are not found. Further in the investigation of the computational complexity of the GM algorithm (\cite{G1}), upper bounds for the integers $n_i$ are found, but the integers themselves are not computed. In this paper, we actually compute the integers.
\end{rem}
\begin{rem} We remind the reader of the difference between an algorithm and a procedure.
An algorithm is a procedure that always stops and gives an answer.
     The proof of  the existence of a geometric algorithm (an algorithm as opposed to a procedure)  to determine discreteness or non-discreteness proceeds using J{\o}rgensen's inequality to obtain a positive lower bound for the difference between the trace of $A$ and that of $AB^{-1}$ when the pair $(A,B)$ is coherently oriented\cite{GM}. Thus in replacing a pair $(A,B)$ by $(B^{-1}, A^{-1}B^n)$ we can show that we have decreased the trace by a positive amount.     %{\color{red}FIX}
      The concept of trace minimizing is due to Purzitsky and has been further developed by Purzitsky and Rosenberger
  \cite{P, PR}.
\end{rem}
     We first state the main result in the case that the initial and stopping generators are hyperbolic isometries with disjoint axes as this is the most complicated case of the algorithm and then give extensions allowing parabolics; in section \ref{section:all} we indicate how this extends to all other type of pairs of initial and stopping generators. We use $T_X$ and $K_X$ to denote the translation length and multiplier of an element $X \in PSL(2,\mathbb{R})$ or equivalently the translation length and multiplier of a lift of $X$ to $SL(2,\mathbb{R})$ as defined in section \ref{section:NOTe}, where $\Tr(X)$, the trace of $X$ or a lift is also defined.

Recall that a hyperbolic transformation has a {\sl  multiplier}.  We first note that there are several different commonly used definitions of the {\sl multiplier} in the literature, in particular  on page 13 of \cite{Fench}  $m$ is used and on page 5 of \cite{Maskit} where $k^2$ is used.
In the latter which considers complex M{\"o}bius transformations, $k^2$ is defined by
conjugating the fixed points $x$ and $y$ of the transformation to  $0$ and $\infty$ noting that the conjugated transformation $h$  satisfies: $ {\frac{h(z) -x}{h(z)-y}} = k^2{\frac{z-x}{z-y}}$ and defining $k^2$ to be the multiplier assuming $x \ne y$ so that the transformation is conjugate to $z \mapsto k^2z$, $z \in {\overline{\mathbb{C}}}$, with {\it $y$ the attracting fixed point and $x$ the repelling fixed point}.
In \cite{Fench} a transformation is considered to have two multipliers, $m$ and ${\frac{1}{m}}$ depending on whether $0$ or $\infty$ is used as the attracting fixed point the appropriate conjugate transformation.

Here we use $K$ to denote the multiplier.  Our $K$ is Maskit's $k^2$ and Fenchel's $m^{\pm1}$. Since we are in the real case, $K$ is real and we  may assume that $K =k^2 > 1 $. However, we do not need to distinguish between the two multipliers $m$ and ${\frac{1}{m}}$ in our results because in our calculations we always take the absolute value of the logarithm  of the multipliers. We have ${\tilde{X}}$ conjugate to $z \mapsto Kz$ for some real number $K>0$. We note that $|{\Tr(\tilde{X})}|= \sqrt{K} + \sqrt{K}^{-1}$. We write $K_X$ for the multiplier of $\tilde{X}$ and note that $T_X = T_{X^{-1}} = |\log K_X| $ via $\cosh {\frac{T_X}{2}} = {\frac{1}{2}} \Tr(X)$. Note also that $K_{X^{-1}} = K_X^{-1}$. We do note that the above facts imply  if an ordered pair $(A,B)$ is {\sl coherently oriented} then  $K_A >  K_B >1 $.

\begin{thm} \label{theorem:new}{\rm (hyperbolic-hyperbolic initial and stopping generators)}
Assume that  $A$ and $B$ are a pair of hyperbolics with disjoint axes and the algorithm stops with such a pair. Interchange $A$ and $B$ and replace $A$ and/or $B$ by their inverses as necessary so that we may assume that the initial ordered pair $(A,B)$ is coherently oriented.

If one applies the Euclidean algorithm
to the non-Euclidean translation lengths of the generators at each
step, the output is the F-sequence $[n_1,...,n_k]$.

\vskip .1in

In particular if the multiplier of $A$ is $K_A$ and the multiplier
of $B$ is $K_B$ with $K_A \ge K_B$, then

$$n_1= [{\frac{(|\log K_A|)/2}{(|\log K_B|)/2}}]$$ where $[\;\;]$ denotes the greatest integer function and $|\;\;|$ absolute value, \\ or equivalently if $T_X$ is the translation length of $X$:
$$n_1 = [{\frac{T_A/2}{T_B/2}}].$$
and
$$n_2 = [{\frac{T_B/2}{T_D/2}}]\;\;\; \mbox{where} \;\;\; D = A^{-1}B^{n_1}$$ and
$$n_j = [{\frac{T_{C_j}/2}{T_{D_j}/2}}]$$
where $(A,B)=(C_1,D_1)$ and $(C_j,D_j) = (D_{j-1}^{-1}, C_{j-1}^{-1}D_{j-1}^{n_j-1})$ is the ordered pair of generators at step $j$, $1 \le j \le t$ in Theorem \ref{theorem:GMalg}.
\end{thm}
Now we extend this to statements of the theorem if the algorithm encounters parabolics.  Note that a parabolic isometry does not have a multiplier. For any $X$ and $Y$,  let $[X,Y]$ be their multiplicative commutator.
\begin{thm}\label{theorem:pars} {\rm (parabolic elements)}

 Assume that at step $j$, $C_j$ is hyperbolic with  $D_j$  parabolic. Then

$$n_j = [ { \frac{\Tr (C_j)-2}{ {\sqrt{|{\Tr ([C_j,D_j])-2}|} } }} ].$$
%\end{thm}
%\end{document}
or equivalently, $$n_j = [ {\frac{2 \cosh({\frac{T_{C_j}}{2}})-2}{\sqrt{2\cosh{( {\frac{T_{[C_j,D_j]}}{2}})-2} }}} ].$$
 If  $C_{j}$ and $D_{j}$ are both parabolic,  $n_j =1$.
 \end{thm}
 If one of $C_j$ or $D_j$ is elliptic, the group is either not free or not discrete.
\section{Further Notation, Terminology  and Preliminaries } \label{section:NOT}
We follow the notation of \cite{Fench}. For  $x$ and $y$ distinct points in  ${\overline{\mathbb{D}}}$, the closure of $\mathbb{D}$, let $[x,y]$ denote the unique geodesic through $x$ and $y$ interior to  ${\overline{\mathbb{D}}}$.  If $a$ and $b$ are the points where ${\overline{[x,y]}}$, the closure of $[x,y]$ intersects the boundary of $\mathbb{D}$, $a$ and $b$ are called the {\sl ends} of $[x,y]$.
\subsection{Factorization}\label{section:factor}
If $M$ is a geodesic in $\mathbb{D}$, we let $H_M$ denote reflection in $M$, that is the element of order two  that fixes $M$ and its ends.  Now any M{\"o}bius transformation, $X$, can be factored in many ways as the product of two reflections.  In the hyperbolic case this is as the product of any two reflections  about perpendiculars to $Ax_X$ that intersect $Ax_X$ at a distance of ${\frac{\T_X}{2}}$ apart.

If $A$ and $B$ are any two isometries of $\mathbb{D}$ (or equivalently $\mathcal{U}$), their axes have a common perpendicular $L$ and there are lines $L_A$ and $L_B$ such that
$A= H_L \circ H_{L_A}$ and $B = H_L \circ H_{L_B}$. The reflection in  $L$ will fix the axis of $A$ and interchange it ends (see Figure \ref{figure:Initial Axes}). Similarly, reflection in $L$ will fix the axis of $B$ and interchange its ends. The group $G = \langle A, B \rangle$ is a subgroup of $\langle L, L_A, L_B \rangle$ of index two so that both groups are simultaneously discrete or non-discrete or not free.

The distance between $L$ and $L_B$ along the axis of $B$ is $T_B/2$. We consider $L_{B^2}, L_{B^3}, ...., L_{B^{r-1}}, L_{B^r}$ successive perpendiculars to the axis of $B$ that are each a distance $T_B/2$ from the preceding one along the axis of $B$ with $L_{B^q}$ separating $L$ and $L_{B^{q+1}}$ for each integer $1 < q < r$. We let $L_B= L_{B^1}$. Then  $H_L \circ H_{L_{B^q}} = B^q$. We further  note that $B$ can be factored as $H_{L_{B^q}}H_{L_{B^{q+1}}}$ for any integer $q$ and that  $H_{L_{B^q}} \circ H_{L_{B^s}} = B^{q-s}$ for integers $q$ and $s$.

The fixed point of an elliptic element lies interior to $\mathbb{D}$ and that of a parabolic lies on the boundary of $\mathbb{D}$. These points are considered the axes of the elliptic or parabolic. Elliptic and parabolic elements can also be factored as products of reflections that fix their axes.
\subsection{Conventions for Figures}
Our results are stated in terms of matrices in $SL(2,\mathbb{R})$ and $PSL(2,\mathbb{R})$. However, figures are drawn in $\mathbb{D}$ where the symmetry is more apparent.
 Note that figures are {\sl schematic} drawings. All lines shown interior to the disc represent geodesics, that is, arcs of Euclidean circles perpendicular to unit disc. The intersection a of geodesic with the boundary of the unit disc is understood to be perpendicular but is not marked as such in order to avoid clutter. However, interior to the disc, perpendicular intersection points of two geodesics  are marked with a blue circle. Usually intersections between a solid and a dotted line of the same color are perpendicular and are marked as such, but intersections between  other lines of different colors whether they are solid or dotted are not assumed perpendicular unless they are marked as such.
\begin{figure}[t] \begin{center}
\includegraphics[width=6cm,keepaspectratio=true]{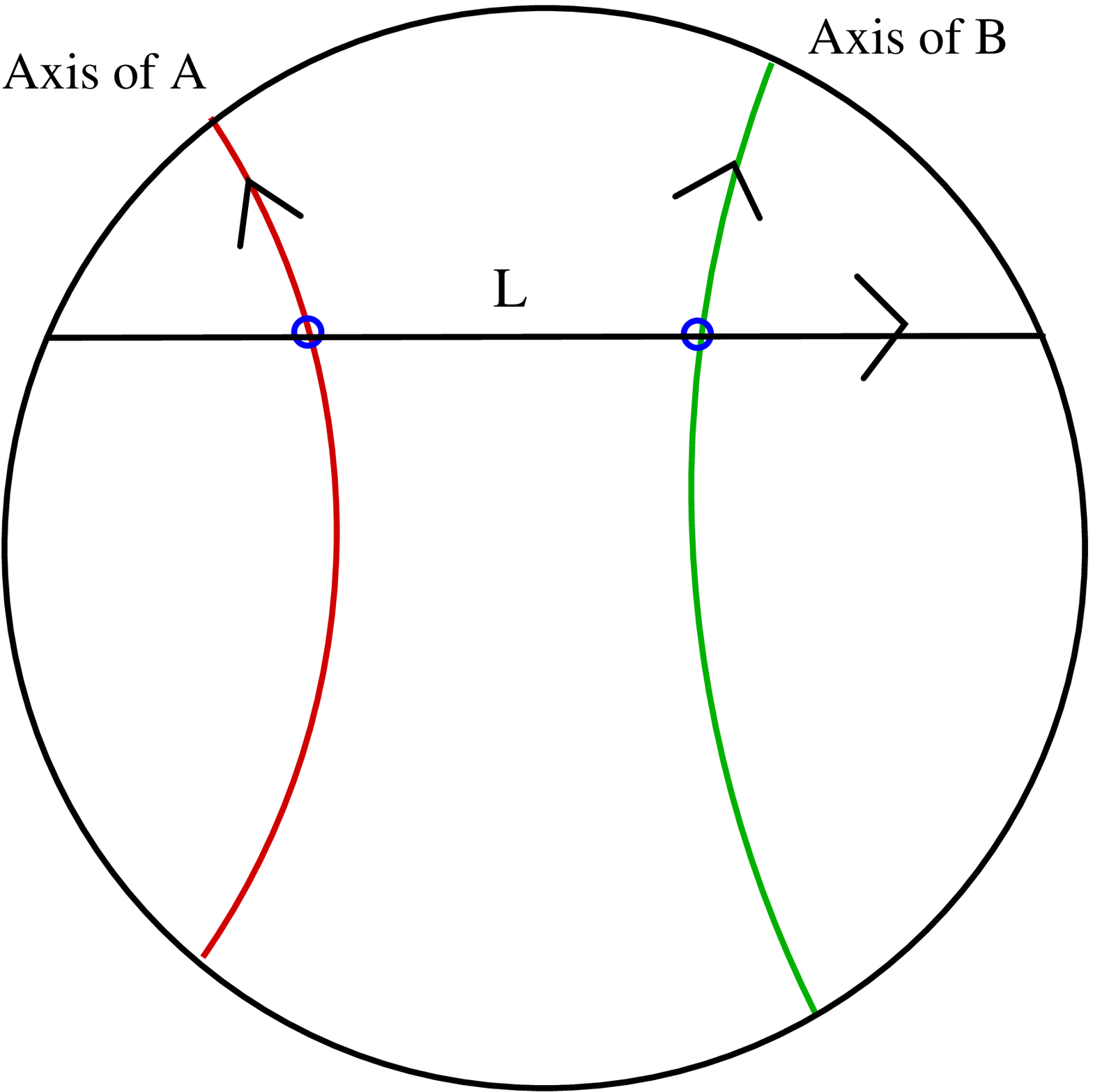}
\includegraphics[width=6cm,keepaspectratio=true]{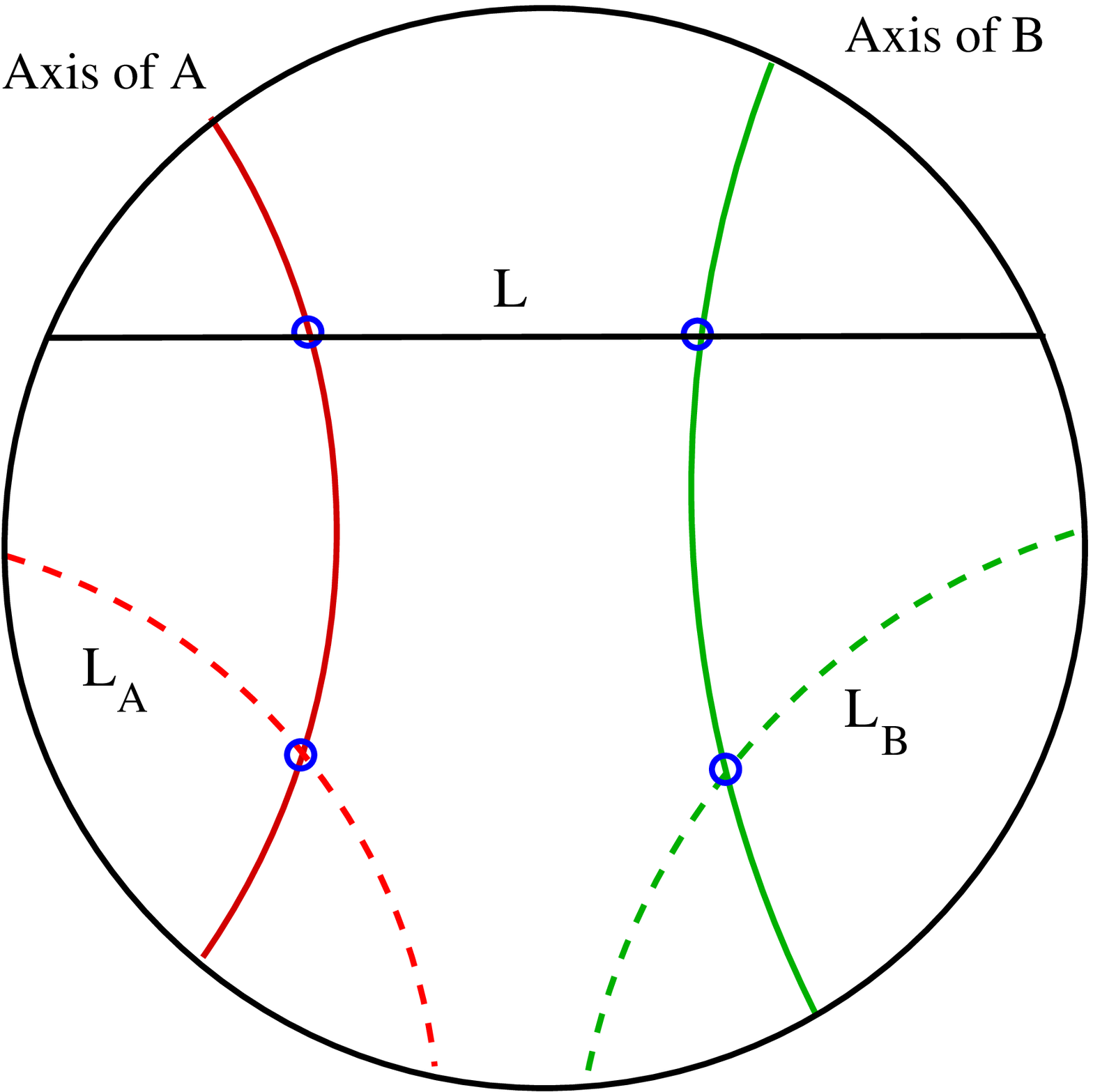}
\vskip .2in
\includegraphics[width=6cm,keepaspectratio=true]{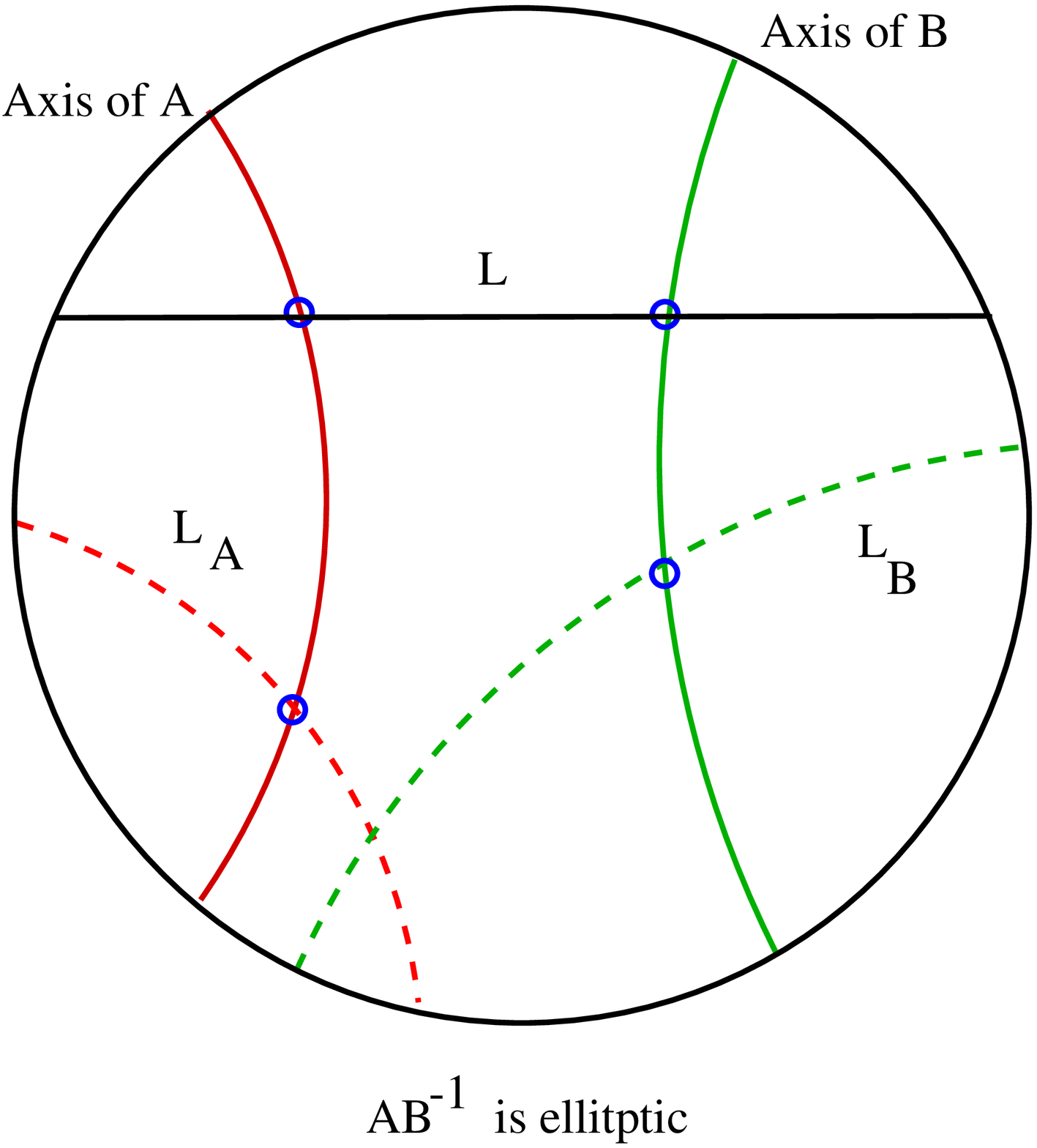}
\includegraphics[width=6cm,keepaspectratio=true]{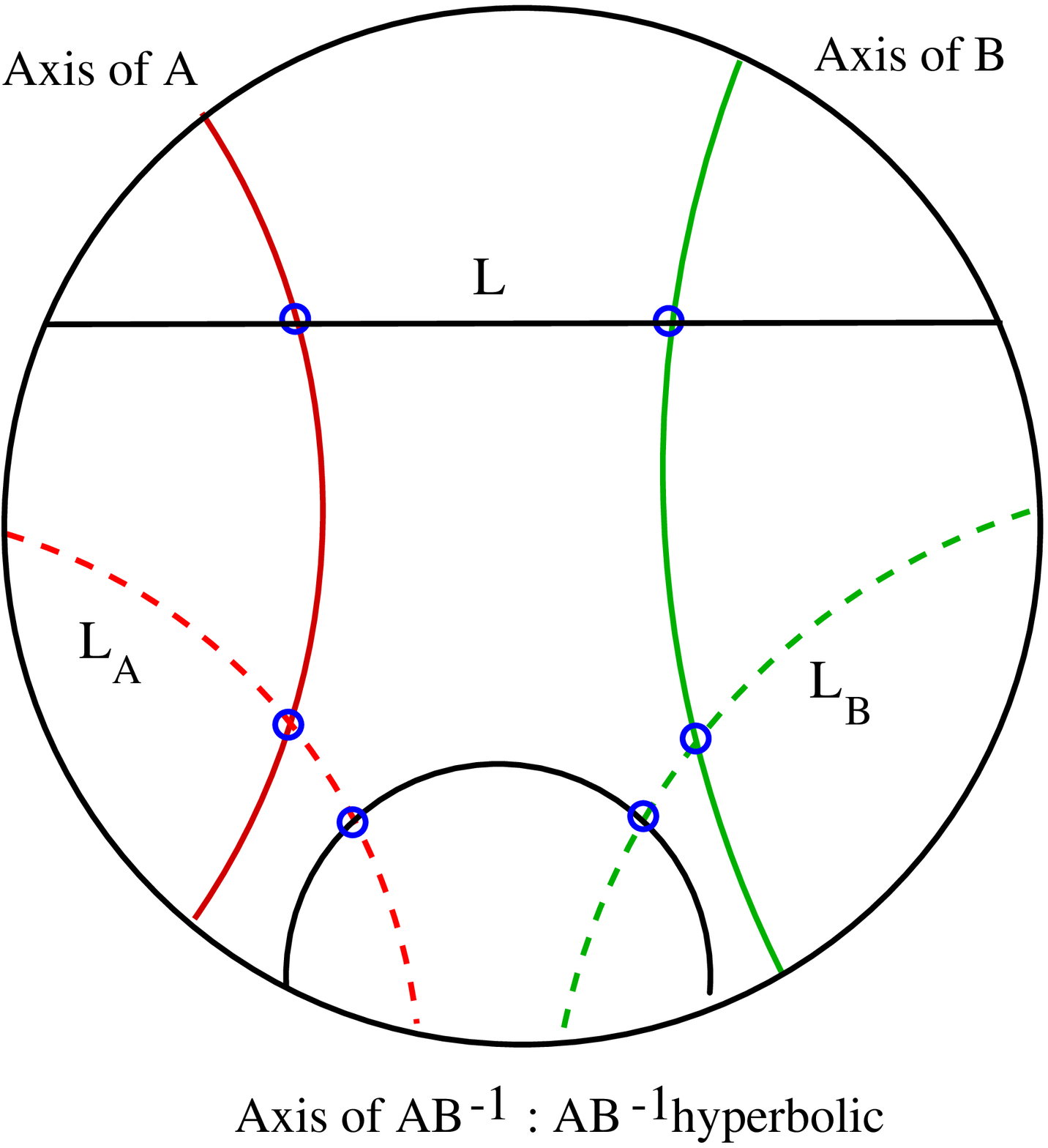}

\caption{Axes of $A$ and $B$ and their common perpendicular $L$ and some possible configurations for the $L_A$ and $L_B$ lines. In the second case the group will be discrete and free. In the third case,  $A^{-1}B$ is elliptic and in the last figure the axes of $AB^{-1}$, the common perpendicular to $L_A$ and $L_B$,  is shown.}
%NNtheFirst; NPossConfig1; Nelliptic; NNtheFourth}
\label{figure:Initial Axes} \end{center} \end{figure}
\begin{figure}[t]
\begin{center}
\includegraphics[width=6cm,keepaspectratio=true]{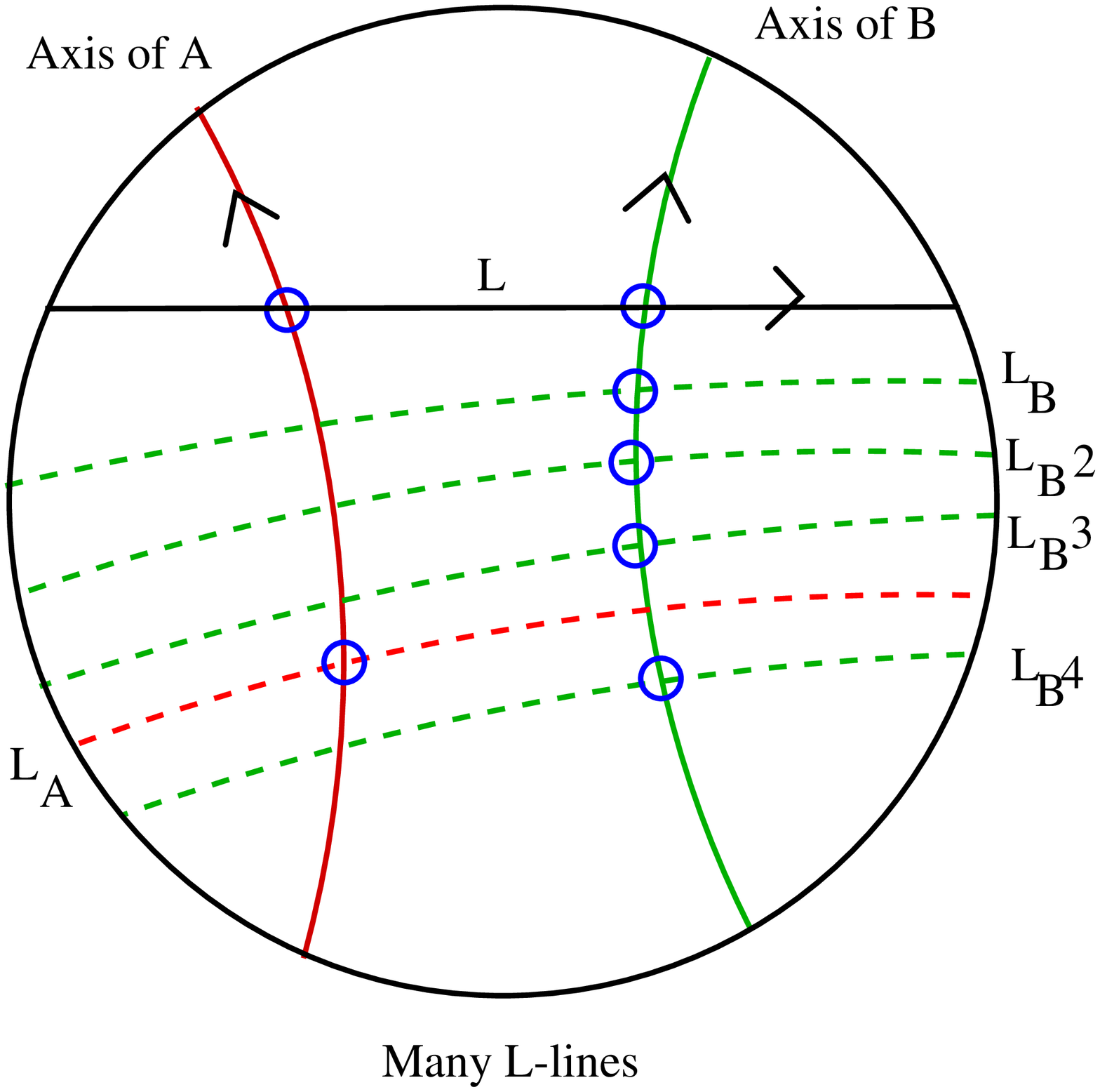}
\includegraphics[width=6cm,keepaspectratio=true]{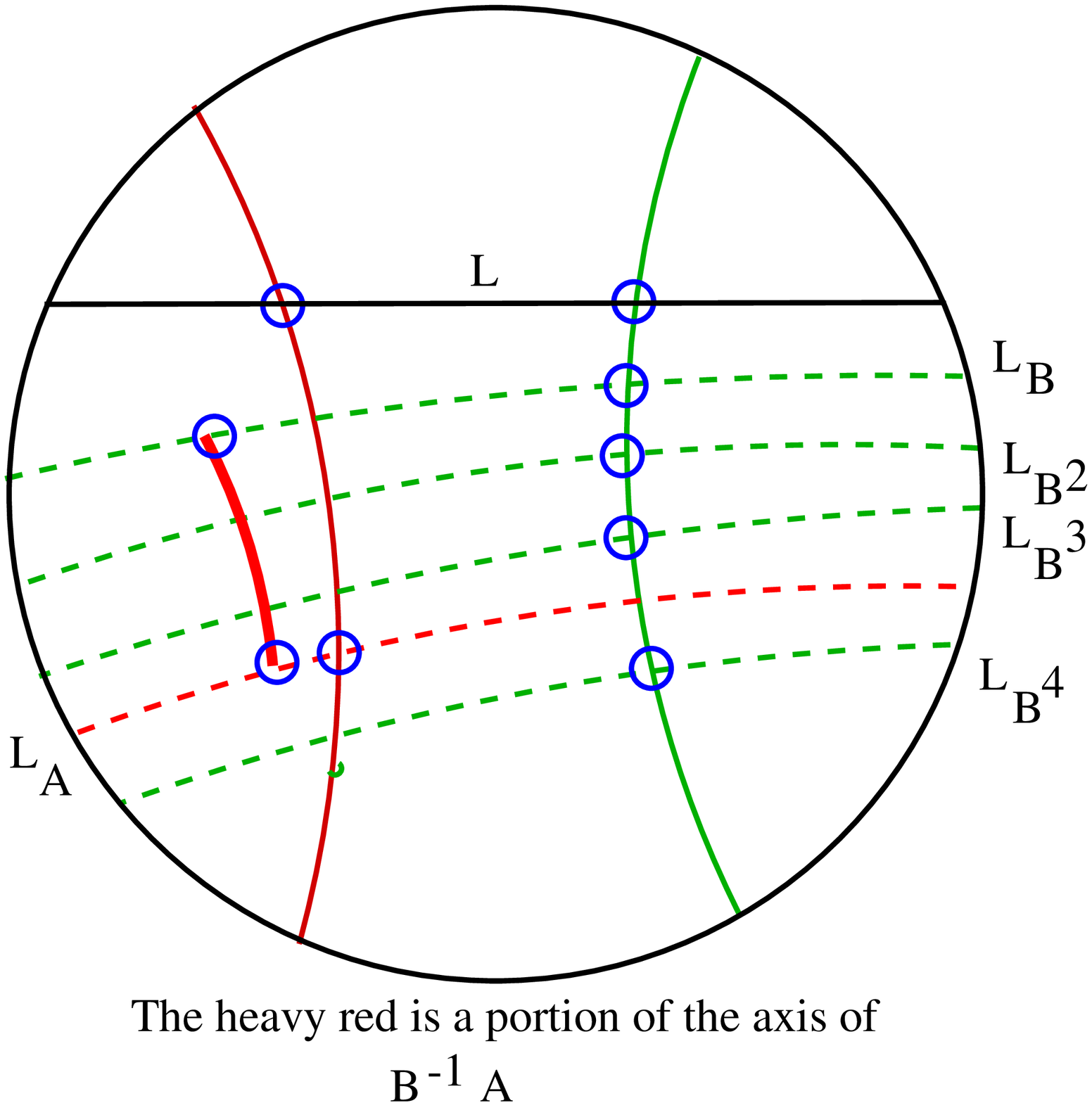}
\vskip .2in
\includegraphics[width=6cm,keepaspectratio=true]{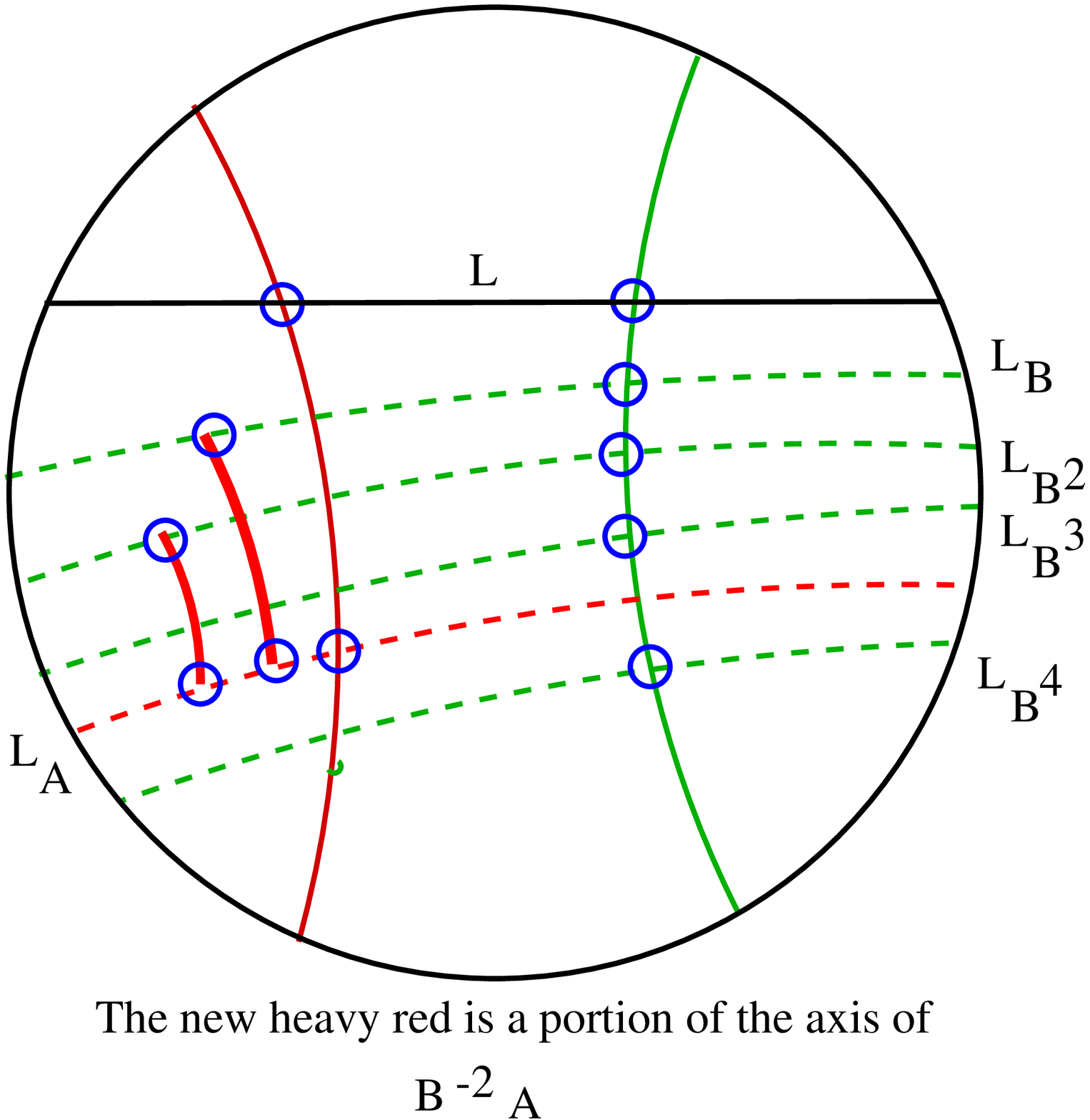}
\includegraphics[width=6cm,keepaspectratio=true]{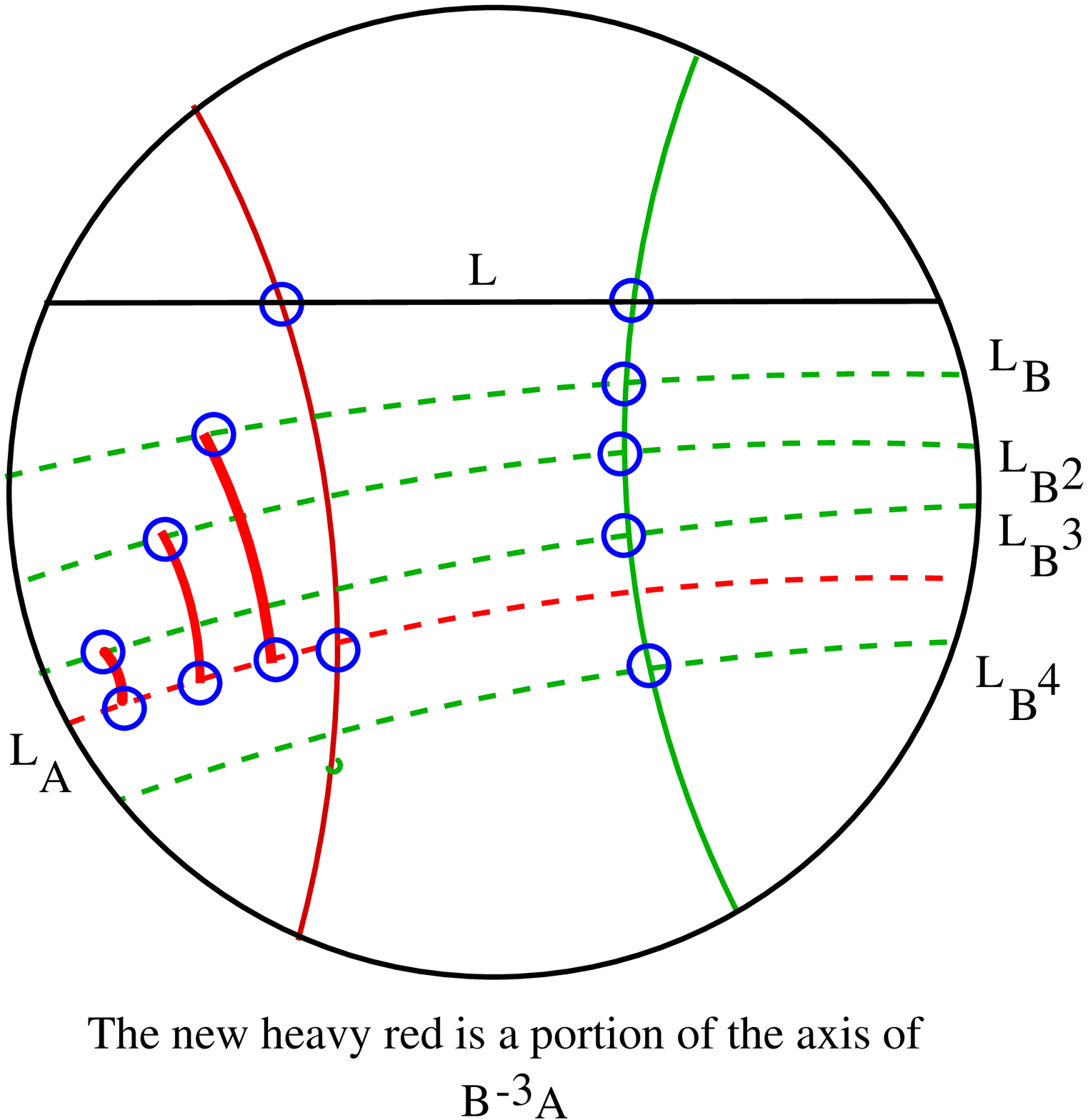}
\caption{A number of $L_B$ lines with $L_A$ intersecting the axis of $B$. Dotted green lines represent $L_B$, $L_{B^2}$, $L_{B^3}$,$L_{B^4}$ respectively. Axes of $A$, $B^{-1}A$, $B^{-2}A$,and $B^{-3}A$ are solid red lines.}
\label{fig:manyL}
\end{center}
                        \end{figure}
% NmoreL, Nonemore,NNtwomore,NNthreemore
%% ManyL/MYTRY; Nonemore; Ntwomore; Nthreemore}
\subsection{The Barebones Algorithm} \label{section:barebones}
Assume we begin with $A$ and $B$ hyperbolics with disjoint axes and orient $L$  so that it  points from  $Ax_A$ to $Ax_B$ when $\Tr (A) \ge \Tr (B) > 0 $ and the attracting fixed points of $A$ and $B$ are both to the left of $L$. When we have such a configuration we say that $A$ and $B$ are coherently oriented. If $A$ and $B$ are not coherently oriented, we  obtain a coherently ordered pair after possibly interchanging $A$ and $B$, swapping $A$ with $A^{-1}$ and/or $B$ with $B^{-1}$. Theorem \ref{theorem:GMalg} assumes the initial $A$ and $B$ are coherently oriented and but it is shown in \cite{GKwords} that the sequence of pairs of generators given by the GM algorithm are coherently oriented so that no further  swapping or interchanging needs to be done.

The algorithm considers the possible configurations for $L_A$ and $L_B$. For example if these two geodesics  intersect interior to $\mathbb{D}$ or they have a common end, then $A^{-1}B$ is either elliptic or parabolic and the algorithm moves up into what is considered to be  an {\sl easier} case, where the generating pair is either a  hyperbolic-elliptic pair or a  hyperbolic-parabolic pair. If the lines are disjoint and no one separates the other from $L$, then the three axis $L$, $L_A$ and $L_B$ bound a region and the algorithm stops indicating  that the group is discrete.

In Figure \ref{figure:Initial Axes} we show some, but not all, of the possible configurations for the $L_B$ and $L_A$ lines. We emphasize those that are relevant to our discussion. for example, if $L_A \cap L_B \ne \emptyset$ and lies interior to the unit disc, then $A^{-1}B$ is elliptic so the group even if discrete, is not free. If these lines intersect on the boundary of the disc, then $A^{-1}B$ is parabolic.

If $L_B$ intersects $Ax_A$ and $L_A$ intersects $Ax_B$, we need to know for how many positive integers $q$ does $L_{B^q}$ intersect $Ax_A$ between $L$ and $L_A$. %where $B^q = H_L H_{L_{B^q}}$.
 Analysis of this situation is carried out in Section \ref{section:mainlemma}
(see Lemma \ref{lem:main}). This is termed the {\sl really bad case} and is the focus of most of our attention. Figure \ref{fig:manyL} illustrates the case when $q=3$ and shows portions of the axes of $B^{-1}A, B^{-2}A, \mbox{ and } B^{-3}A$. The new ordered pair in the algorithmic sequence would be $(B^{-1}, A^{-1}B^3)$.

\subsection{The algorithmic path}\label{section:path}
The algorithm considers the different possible types of generating pairs and breaks them down into two cases. One is the case where both generators are hyperbolic and their axes intersect \cite{G2}. The other case where the axes of the hyperbolic are disjoint is known as the {\sl intertwining} case \cite{GM}.

In the intertwining case, the algorithm tests the generators for discreteness and non-discreteness. If neither is found to be true, the algorithm  declares the pair indeterminate
and produces a {\sl next pair} of generators. We have $(C_1,D_1) = (A,B)$. The algorithm returns  {\sl coherently oriented} pairs as the next pair so that we always have $\Tr (C_j) \ge \Tr (D_j)$ at the $j$-th step. An implementation of the GM algorithm  will begin and end with a pair in any part of the  possible paths given below, but the {\sl next pair}  will  follow the path staying stationary or moving to the right, allowing at most a  finite number of repetitions of a given pair-type before it moves to the right.

 Note that in the GM algorithm we look at $$(A,B), (B^{-1}A,B), (B^{-2}A,B),\dots,  (B^{-(n-1)}A,B)$$ and then move to $(B^{-1},A^{-1}B^n)$. A step of the form $(A,B) \rightarrow  (B^{-1}A,B)$ is known as a linear step and one of the form $(A,B) \rightarrow (B^{-1}, A^{-1}B)$ is known as a Fibonacci step. However, here we note that we can (and do) shorten the algorithm by omitting
 the linear steps. This is done once we have computed the number of $n-1$ of linear steps before a Fibonacci step occurs.

 Let $H$,$P$ and $E$ denote respectively  a hyperbolic, parabolic, or elliptic generator. The types of pairs in the intertwining algorithmic  path are:
        \begin{enumerate}
\item
$H \times H \rightarrow H \times E \rightarrow P \times  E  \rightarrow E \times E$
\item
$H \times H \rightarrow H \times P \rightarrow H \times E \rightarrow  P \times E \rightarrow E \times E$
\item
$H \times H \rightarrow H \times P \rightarrow P \times  E  \rightarrow E \times E $
\item
$H \times H \rightarrow H \times P \rightarrow P \times  P  \rightarrow P \times E  \rightarrow E \times E$
\end{enumerate}
\section{Configurations and Distances: Main Lemma} \label{section:mainlemma}

We turn our attention to the {\sl really bad} case.
\begin{lem} \label{lem:main} {\rm (Main Lemma)}
If $(A,B)$ satisfy  $\Tr (A) \ge \Tr(B) >2$,  $L_A \cap Ax_B \ne \emptyset$, and $L_B \cap AX_A \ne \emptyset$. Let $L$ be the common perpendicular to $Ax_A$ and $Ax_B$. Let $L_A$ and $L_{B^r}$ be the geodesics so that $A=H_LH_{L_A}$ and $B^r=H_LH_{L_{B^r}}$. Let $n$ be the smallest integer so that $L_{B^n}$ separates $L$ and $L_A$ but $L_{B^{n+1}}$ does not. Then
$$n{\frac{T_B}{2}} \le {\frac{T_A}{2}} \le (n+1) {\frac{T_B}{2}}.$$
\end{lem}

\begin{lem}   \label{lem:PPless} Assume $(A,B)$ satisfy  $\Tr (A) \ge \Tr(B) >2$ and  $L_A \cap Ax_B =  \emptyset$, but $L_B \cap AX_A \ne \emptyset$. Let $L$ be the common perpendicular to $Ax_A$ and $Ax_B$. Let $n$ be the smallest integer such that $L_{B^n}$ separates $L$ and $L_A$, but $L_{B^{n+1}}$ does not.   Then  $$n{\frac{(T_B)}{2}} \le {\frac{T_A}{2}}.$$
\end{lem}

\begin{lem} \label{lem:mainpar}
If $\Tr (A) >  \Tr(B) = 2$, with $L_B \cap Ax_A \ne \emptyset$, let $q$ be the smallest s a positive integer such that $L_{B^q} \cap Ax_A \ne \emptyset$, but $L_{B^{q+1}} \cap L_A = \emptyset. $ Then
$$q =  [ { \frac{\Tr (A)-2}{ {\sqrt{|{\Tr ([A,B])-2}|} } }} ].$$
%$q =  [  {\frac{\Tr (A)}-2}{{\sqrt{\Tr ([A,B])-2} } }} ].$
If $\Tr (A) = \Tr(B) =2$, then  if $L_A \cap L_B =  \emptyset$ the group is discrete and if  $L_B \cap L_A \ne \emptyset$, then the group is either not free or not discrete. The $F$-sequence ends.
\end{lem}
\begin{proof} (Proof of Lemma \ref{lem:main}.)

-We have the following order of points moving along the Axis of $B$ toward the repelling fixed point of $B$
$$Ax_B \cap L, Ax_B \cap L_B, \dots , Ax_B \cap L_{B^n}, Ax_B \cap L_A, Ax_B \cap L_{B^{n+1}}$$
and the order of points along the Axis of $A$ towards the repelling fixed point of $A$ is
$$Ax_A \cap L, Ax_A \cap L_B, \dots , Ax_A \cap L_{B^n}, Ax_B \cap L_A.$$
-Since distances along geodesics are additive we have
$$nT_B/2 = \rho(L \cap Ax_B, L_{B^n} \cap Ax_B) \le \rho(L \cap Ax_B, L_A \cap Ax_B) $$
 and
 $$\rho(L \cap Ax_B, L_A \cap Ax_B) \le \rho(L \cap Ax_B, L_{B^{n+1}}\cap Ax_B) = (n+1)T_B/2.$$
-Since the arc of the common perpendicular between any two geodesics has the shortest length of any geodesics arc between the two geodesics, we have
 $$nT_B/2 = \rho(Ax_B \cap L, Ax_B \cap L_{B^n}) \le \rho(Ax_A \cap L, Ax_A \cap L_{B^{n}})$$
 $$ \le
\rho(Ax_A \cap L, Ax_A \cap L_A) = T_A/2$$
Further $$T_A/2 = \rho(L \cap Ax_A, L_A \cap Ax_A) \le \rho(L \cap Ax_B, L_A \cap Ax_B)$$ $$ \le \rho(L \cap Ax_B, L_{B^{n+1}} \cap Ax_B )= T_{B^{n+1}}/2.$$
 Finally,  suppose  for some integer $r$,  $L_{B^r} \cap Ax_A = L_A \cap Ax_A$, then  $AB^{-r}$ is elliptic so the group is  not discrete and free. If $L_A \cap Ax_B = L_{B^r} \cap Ax_B$,  then again the group has an elliptic element and is not discrete  and free.
\end{proof}
As a corollary we have
\begin{cor} For $n$ as in Lemma \ref{lem:main}, $T_{AB^{-n}} \le T_B$.
\end{cor}
\begin{proof} Continuing with the notation of the proof of Lemma \ref{lem:main}
we  have
$$T_{AB^{-n}}/2 \le \rho(Ax_B \cap L_A, Ax_B \cap L_{B^n})\le \rho(Ax_B \cap L_B^n, Ax_B \cap L_{B^{n+}}) = T_B/2.$$
\end{proof}
The proof of Lemma \ref{lem:PPless} only requires minor modifications from that of Lemma \ref{lem:main} and the integer $n_j$ at step $j$ is thus the  $[{ \frac{T_A/2}{T_B/2}}]$ where $A=C_j$, and $B=D_j$.
\begin{proof} (Proof of Lemma \ref{lem:mainpar}.)
 If the algorithm begins or encounters a parabolic, the final $n_t$ has a different definition. At each algorithmic step (see \cite{GM} or \cite{G1}) the  trace of the hyperbolic in the pair $(A,B) = (C_j,D_j)$ where $A$ is hyperbolic and $B$ parabolic is reduced by a fixed amount, which can be computed to be $|{\sqrt{\Tr([A,B])-2}}|$. To see this  normalize the matrices by conjugation
so that $A(z) = z + \tau$ and $B(z) = {\frac{az+b}{bz+a}}$ where $a^2-b^2=1$ to see that
the amount by which the trace is reduced can be
 written as $|\tau \cdot b|$, that is,  by $|\Tr A - \Tr AB^{-1}|$.  Also compute that $\Tr ([A,B]) = 2 + \tau^2b^2$.  Then $n_t$ is as in the statement of the Lemma.

If both $A$ and $B$ are parabolics, then by the GM algorithm either the product is elliptic and the group is not free or the product is hyperbolic or parabolic and the group is discrete so $n_j =1$.
\end{proof}

\begin{rem} When $C_j$ is hyperbolic and $D_{j}$ is parabolic, the $n_j$ found here
is the same integer as that found in Theorem 4 of \cite{P}. That theorem has different hypotheses, but the hypotheses used here are different.

Pairs of parabolics is one type of pairs of generators for which there is a complete non-algorithmic solution to the discreteness problem \cite{Beard2}. The reader should see also \cite{Parker} which  gives a complete solution to the parabolic-elliptic case and a partial solution to the elliptic-elliptic case along the same lines as \cite{Beard2}.
\end{rem}
\begin{rem}
 We note that if $T_A = T_B$, then either $L_A$ and $L_B$ intersect, in which case $AB^{-1}$ is elliptic so that $G$ is either not free or not discrete or $L_A$ and $L_B$ are disjoint in which case the group is discrete.
  \end{rem}

\begin{rem}
As shown in \cite{GM},
the algorithm stops in the if either $\Tr (C_jD_j^{-1}) \le  -2 $ in which case the group is discrete and free, $-2 < \Tr (C_jD_j^{-1} < 2$ so that the group is not free or not discrete, or $(C_j,D_j)$ violates J{\o}rgensen's inequality (e.g. $|\Tr([C_j,D_j])-2| + |\Tr^2( C_j) -4| \ge 1$).
Further if J{\o}rgensen's inequality holds, then in the $H \times H$ case, $\Tr A - \Tr (AB^{-1})$ has a positive lower bound (i.e. ${\frac{({\sqrt{2}}-1)^2}{\sqrt{2}}}$). This  assures that the number of linear steps before a Fibonacci step is finite.
\end{rem}

\begin{proof} (Proof of Theorem \ref{theorem:new}.)
Assume that at some point the algorithm returns the pair $(A,B)$. If $L_A$ and $L_B$ intersect, then the algorithm is stopped and $G$ is either not free and or not discrete. If the three geodesics $L$, $L_A$ and $L_B$ bound a region, $G$  is discrete. If J{\o}rgensen's inequality is violated, $G$ is not discrete. Otherwise, we are in one of the cases of one of the three lemmas and the result follows.
\end{proof}

Combining the results we have
\begin{thm} \label{theorem:HP}
Let $G = \langle A, B \rangle$ with $(A,B)=(C_1,D_1)$  and let  $(C_t,D_t)$ be the stopping pair for the GM algorithm.

(1) Assume that $A$ and $B$ are a pair of hyperbolics with disjoint axes and that the GM discreteness algorithm stops with such a pair.

If one applies the Euclidean division algorithm
to the non-Euclidean translation lengths of the generators at each
step, the output is the F-sequence $[n_1,...,n_t]$.
\vskip .05in
In particular if the multiplier of $A$ is $K_A$ and the multiplier
of $B$ is $K_B$, then
$$n_1= [{\frac{(|\log K_A|)/2}{(|\log K_B|)/2}}]$$ where $[\;\;]$ denotes the greatest integer function and $|\;\;|$ absolute value, \\ or equivalently if $T_X$ is the translation length of $X$:
$$n_1 = [{\frac{T_A/2}{T_B/2}}]$$
and
$$n_2 = [{\frac{T_B/2}{T_D/2}}]\;\;\; \mbox{where} \;\;\; D = A^{-1}B^{n_1}$$ and
$$n_j = [{\frac{T_{C_j}/2}{T_{D_j}/2}}]$$ where $(C_j,D_j)$ is the ordered pair of generators at step $j$, $1 \le j <  t$ in Theorem \ref{theorem:GMalg}.

(2) If at step $j$ $(C_j,D_j)$ %If at step $t-1$,  $(C_{t-1},D_{t-1})$
 is  hyperbolic-parabolic pair, then
 $$n_j = [  { \frac{{\Tr (C_j)}-2}{ {\sqrt{|{\Tr ([C_j,D_j])-2}|} } }} ]$$ or equivalently, setting $A=C_j$ and $B = D_j$.
 $$n_j = [ {\frac{2 \cosh({\frac{T_A}{2}})-2}{\sqrt{2\cosh{( {\frac{T_{[A,B]}}{2}})-2} }}} ].$$

(3)  If  at step $j$ $(C_j,D_j)$ is a parabolic-parabolic pair, $j = t$ and  $n_t =1$.
\vskip .05in

Further if $t \ge 2$, at step $(t-1)$,  $(C_{t-1},D_{t-1})$ is  a hyperbolic-parabolic pair with   $$n_{t-1}= [  { \frac{{\Tr (C_{t-1})}-2}{ {\sqrt{|{\Tr ([C_{t-1},D_{t-1}])-2}|} } }} ]$$
and$$ n_t=1.$$

(4) If the initial pair is a hyperbolic-parabolic pair, then the $F$-sequence if of length $2$ and is $[n_1,n_2]$
where
$$n_1 = [ {\frac{2 {\cosh({\frac{T_A}{2}})}-2}{\sqrt{2\cosh{( {\frac{T_{[A,B]}}{2}})-2}} }}]$$
 and $$n_2=1.$$

(5) If the initial pair is a parabolic-parabolic pair, the $F$-sequence is of length $1$ with $n_1=1$.
  \end{thm}

\section{The analogy: the Euclidean and the non-Euclidean algorithm} \label{section:analogy}

We find the sequence $[n_1,...,n_t]$ of the discreteness algorithm
by a combination of Euclidean or division algorithm type of computations with
hyperbolic lengths and hyperbolic length replacements with the
remainder term (in keeping with the trace reducing aspect of the
algorithm).

\vskip .2in  Let $a= (|\log K_A|)/2$ and $b= (|\log K_B|)/2$.

\medskip
Then these are simultaneously hyperbolic lengths and real numbers.
We can do the first step of the Euclidean type algorithm on $a$ and $b$
to obtain (assuming $a > b$):
\bigskip

{\bf Step 1.}  $a = n_1b + b_1$ where $0 \le b_1 < b$ and $n_1$ is a
positive integer.

\bigskip

{\bf Step 2.}  At step 2, in a standard Euclidean algorithm, we
would normally work with $b_1$ and $b$:

\smallskip

{\bf BUT} we replace $b_1$ by $T_D/2 = |\log K_D|/2 $. (Note: we know
from the geometry $T_D/2 \le b_1$.)

\smallskip
That is,  set $D= A^{-1}B^{n_1}$ and use ${\tilde{b_1}} = T_D/2$.
\bigskip

{\bf Subsequent Steps.} Next,  $b = n_2{\tilde{b_1}} + b_2$ where $0
\le b_2 < b_1$ and $n_2$ is an integer.
\smallskip
Replace $b_2$ by $\tilde{b_2}= T_E/2$ where $E =B(A^{-1}B^{n_1})^{n_2}$ and continue. Note that $T_E = T_{E^{-1}}.$

\smallskip {\bf Stop.} The geometric proof that the algorithm stops in a
finite number of steps (b/c the trace is reduced at each step by at
least a minimal amount) says  that here after a finite number of
steps, we are at a stopping point and is recognized by a trace that is less than $2$ (as the initial traces are taken to be positive). Thus the algorithm stops if either the group contains an elliptic (recognized by trace between -2 and 2), if the reflection axes bound a region (recognized by a trace that is less than or equal to $-2$), or if J{\o}rgensen's inequality has been violated.
\section{Including elliptics and intersecting axes} \label{section:all}

In the case that the algorithmic path encounters an elliptic, either the elliptic is of infinite order in which case the group is not discrete or it is of finite order in which case the algorithm proceeds until it reaches a decision about discreteness. The $F$-sequence needs to be modified (see \cite{Vidur}) because the algorithm uses both hyperbolic distances and angles. Angles are the same for Euclidean and non-Euclidean geometry. The interpretation as a non-Euclidean Euclidean algorithm can be continued, but we do not do so here.

For hyperbolics with intersecting axes, one does not need an algorithm to determine discreteness unless the commutator is elliptic. However, the same rule for replacing generators using the $F$-sequence will stop and find the shortest generators.

Thus whether we are in the intertwining case or the intersecting axes case, if the group is discrete and free so that the quotient is a surface, the steps used in the algorithm can be applied and interpreted as an algorithm to find the three shortest geodesics on the quotient
(see \cite{GKwords, HYPhand}). This algorithm  can be interpreted as a non-Euclidean Euclidean algorithm with the same definitions of the $n_i$. We have:

 \begin{cor} {\bf (Shortest Curves)} {\rm (initial generators hyperbolics with intersecting or disjoint axes)} \label{cor:shortest}
 Given $G = \langle A, B \rangle \subset PSL(2,\mathbb{R})$ where $A$ and $B$ are  hyperbolic either with   disjoint or intersecting axes, there is an non-Euclidean Euclidean Algorithm  that finds the  generators corresponding to the three shortest curves on the quotient surface when the group is discrete and free. There is an  $F$-sequence of integers and these integers are calculated by the same formulas as the formulas in Theorem \ref{theorem:new}.
 %{\color{red} length zero? can int axes proceed to par by par? JG check}
 \end{cor}

\smallskip
{\bf Acknowledgements.} The author would like to thank Bill Goldman, Ryan Hoban, and Sergei Novikov for some useful and thought provoking conversations and e-mail exchanges. Thanks are also due to the Mathematics Department at the CUNY Graduate Center for its hospitality during the author's sabbatical. This work was partially supported by the NSF when the author was on an IPA appointment which included released time for research. The NSF has no responsibility for the content.

\smallskip

The author thanks the referee for numerous helpful suggestions and comments.

\bibliographystyle{amsplain}

%\bibliography{99}

%\bibliographystyle{amsplain}

\end{document}